\documentclass[11pt]{article}

\usepackage{amsmath,amssymb,amsthm}
\usepackage[english]{babel}

\newcommand{\prob}{\mathbb{P}}

\newcommand{\BE}{\mathbb{E}}
\newcommand{\Var}{\operatorname{Var}\nolimits}
\renewcommand{\Re}{\operatorname{Re}\nolimits}
\newcommand{\BR}{\mathbb{R}}
\newcommand{\CP}{\mathcal{P}}
\newcommand{\CD}{\mathcal{D}}
\newcommand{\rad}{\rho}
\newcommand{\eps}{\varepsilon}
\newcommand{\tphi}{{\widetilde{\varphi}}}
\newcommand{\intfloor}[1]{\left\lfloor #1\right\rfloor}
\newcommand{\ii}{\mathrm{i}}

\newtheorem{thm}{Theorem}
\newtheorem{lemma}[thm]{Lemma}
\newtheorem{proposition}[thm]{Proposition}
\newtheorem{corollary}[thm]{Corollary}
\newtheorem{conjecture}[thm]{Conjecture}
\newtheorem*{conjecture*}{Conjecture}
\theoremstyle{definition}
\newtheorem{definition}{Definition}
\theoremstyle{remark}
\newtheorem*{remark}{Remark}
\newtheorem*{remarks}{Remarks}

\author{Yu.~Yakubovich\thanks{This research was partly supported by the 
President of RF leading science
schools supporting grant NSh-2460.2008.1. Part of this research was made 
during the author's stay at the Erwin Schr\"odinger Institute for Physics
and Mathematics, Vienna, 
during the programme ``Combinatorics and Statistical Physics'' in Spring 2008.}}
\title{Ergodicity of multiplicative statistics}

\begin{document}
\maketitle
\begin{abstract}
For a subfamily of multiplicative measures on integer partitions
we give conditions for properly rescaled associated Young diagrams 
to converge in probability
 to a certain deterministic curve named the limit shape
of partitions.  We provide explicit formulas for the scaling function and the
limit shape covering some known and some new examples.
\end{abstract}


\section*{Introduction}
It is now widely known that a random partition of a large integer taken
with equal probability among all partitions of that integer has the Young
diagram which looks (after rescaling) close to a deterministic object
called a limit shape of random partitions.
The discovery of the phenomenon of limit shape formation 
for random partitions of a large integer 
has a long history. First it was mentioned in the paper by H.~N.~V.\
Temperley~\cite{T} in~1952 with heuristic arguments. 
Much later but independently, the principal calculations leading to this 
result was made by
M.~Szalay and P.~Tur\'an \cite{ST} in 1977, however they did not state their
result in the modern way.  It was done by A.~Vershik and stated in his joint
paper with S.~Kerov~\cite{VK} in~1985.  

Later a new proof of the same result was found based on the fact that
the uniform measures on partitions of $n$ are just the product measures 
on the space of finite sequences
restricted to an affine subspace.  (The probabilist would say that a random
partition of $n$ is just a sequence of independent random variables with
 specific distributions conditioned to have some weighted sum equal~$n$.)
This technique seems to be first applied to random partitions by 
B.~Fristedt~\cite{F}; now it is often referred to 
as Fristedt's conditional device.
Later it has been frequently used by various authors
 in the related problems, see \cite{P,CPSW,GH}, to name just a few references. 
Vershik \cite{V-FAA} noted that the similar technique is applicable to 
a wider range of problems with the same property that the measure is
a product measure restricted to a certain affine subset. Vershik called 
such measures on partitions \textit{multiplicative}; we give the 
precise definition in Section~\ref{sec:multfam}.

Similar ``limit shape type'' results have appeared in diverse contexts
including some probability measures on partitions, both multiplicative and 
not.  One of the first results of this type goes back to a seminal
paper by P.~Erd\H{o}s and J.~Lehner~\cite{EL}: it can be read from their paper
that rescaled Young diagrams of strict partitions of a large integer concentrate
around a certain limit shape.  This is also a multiplicative case, although
Erd\H{o}s and Lehner did not use the related technique.   A.~Comtet et 
al.~\cite{CMOS} recently found the limit shape for the generalization 
of this case, namely for partitions such that the difference 
between parts exceed
some fixed number~$p$. For $p\ge2$ this family is not multiplicative.
 R.~Cerf and R.~Kenyon \cite{CK} 
confirmed Vershik's conjecture that the limit shape exists also for plane 
partitions. 

In his paper \cite{V-FAA} Vershik introduced several families of multiplicative
measures on partitions and stated that for these measures limit shapes also 
appear in the proper scaling. He called such examples \textit{ergodic} (see
Definition~\ref{def:erg} below) and asked a
general question about conditions for ergodicity of multiplicative
measures.  In this note we give a partial answer to this question although
it remains still open in the full generality. Some connections between 
the existence of limit shape and other asymptotic characteristics 
of measures on partitions is discussed in~\cite{EG08}. The
whole family of multiplicative measures is naturally parameterized by a
sequence of functions $f_k$ analytic in some neighborhood of zero, as explained
in Section~\ref{sec:multfam}. We restrict ourself to a subfamily of 
multiplicative measures such that each member of this sequence is
a power of some fixed function: $f_k(z)=f(z)^{b_k}$, $b_k\ge0$, with some
additional constraints on $f$ and sequence $b_k$.  This restriction looks
quite confining at a first glance  but it allows us to find  exact
formulas for the limit shape and to catch what happens in a more general
case. This family includes, for instance, the measures on partitions which
arise in connection with Meinardus' Theorem ($f(z)=1/(1-z)$), 
see~\cite[Ch.~6]{Andrews}. B.~Granovsky et~al.~\cite{GES} applied 
recently a probabilistic technique going back to Khinchin to this problem
and improved Meinardus' results; our approach partly overlaps with one
used in~\cite{GES}.

The lack of natural nonergodic multiplicative examples 
makes it harder to answer Vershik's question.  
Actually, essentially the only well studied example is the so-called
Ewens measure on partitions. We briefly describe its construction, 
see~\cite{ABT} (where the term Ewens sampling
formula is used) for a more detailed exposition.
Take a random permutation $\pi$ from the symmetric group $S_n$ 
with probability proportional to $\theta^{l(\pi)}$ where $\theta>0$ is 
a parameter and $l(\pi)$ is the number of cycles in permutation $\pi$; 
$\theta=1$ corresponds to the uniform measure on $S_n$.
Given $\pi$, consider the partition of $n$ on cycle lengths of $\pi$. 
The induced probability measure on partitions is called the Ewens measure.
 Mapping partition $\lambda=(\lambda_1,\lambda_2,\dots)$ with the usual
order $\lambda_1\ge\lambda_2\ge\dots$ 
to a simplex $\nabla=\{(x_1,x_2,\dots):\sum x_i\le 1\text{ and }x_1\ge 
x_2\ge\dots\}$ by dividing parts of $\lambda$ by its weight $\sum \lambda_i$
induces the sequence of discrete measures  on $\nabla$; taking their weak limit
in the standard topology  leads to the Poisson--Dirichlet measure
$\mathcal{PD}(\theta)$ on $\nabla$. This measure is not concentrated on 
the unique element of~$\nabla$, so the Ewens measure is not ergodic.

In this note we give another examples of nonergodic behavior 
(Proposition~\ref{prop:nonerg}) in a slightly weaker sense.  However all these
examples are degenerate.

The rest of the paper is organized as follows.  In the next section we 
give a precise definition of multiplicative measures on partitions
and present some basic facts about them. In particular, we introduce the 
notions of grand and small canonical ensembles of partitions. 
In the end of the section we formulate
further assumptions we impose on the multiplicative measure in this note.
In Section~\ref{sec:erg} we present a definition of ergodicity and discuss
its basic consequence.  Section~\ref{sec:erglce} is devoted to ergodicity
in the grand canonical ensemble. We give necessary and sufficient conditions
for ergodicity in the considered class of multiplicative measures 
(Theorem~\ref{thm:lceerg}) and
provide an explicit formula for the limit shape in the ergodic case.
In Section~\ref{sec:ergsce} we give sufficient conditions for ergodicity
in the small canonical ensemble.  We conclude the paper with examples in 
Section~\ref{sec:examples}.

\section{Multiplicative families}\label{sec:multfam}

The multiplicative families of measures on partitions were defined in 
\cite{V-FAA} in the following way. For each $n=1,2,\dots$ let $\mu^{(n)}$ 
be a probability measure defined
on a set $\CP(n)$ of integer partitions of $n$. For partition 
$\lambda\in\CP(n)$ define \textit{completion numbers} or \textit{counts} 
$R_k(\lambda)=\#\{i:\lambda_i=k\}$, $k=1,2,\dots$, 
that is the number of parts $k$ in partition~$\lambda$.  Measure $\mu^{(n)}$
makes $R_k$ random variables: $\prob[R_k=j]=\mu^{(n)}\{\lambda\in\CP(n):
R_k(\lambda)=j\}$.  These random variables are obviously dependent since
the relation
\[
N(\lambda):=\sum_{k=1}^\infty kR_k(\lambda)=n
\]
holds for all $\lambda\in\CP(n)$, i.e.\ $\mu^{(n)}$-almost sure.  We also 
introduce a set $\CP(0)=\{\varnothing\}$  and the trivial 
probability measure $\mu^{(0)}$ on it.

\begin{definition}\label{def:mult}
The family of probability measures on partitions $\mu^{(n)}$ is called 
\textit{multiplicative} if there exists a sequence of positive numbers 
$\{\bar a_n\}_{n\ge0}$ such that $\sum_n \bar a_n=1$ and counts $R_k$ 
are mutually independent 
with respect to the convex combination $\bar\mu:=\sum_n \bar a_n\mu^{(n)}$.
\end{definition}

Independence of $R_k$ with respect to $\bar\mu$ means that there exists a
rectangular array of nonnegative numbers $\bar g_{k,j}$, $k\ge1$, $j\ge0$ and
$\sum_{j=0}^\infty \bar g_{k,j}=1$ for any $k$, 
such that for any partition $\lambda\in\CP=\cup_{n=0}^\infty \CP(n)$
\begin{equation}\label{eq:defmu1}
\bar\mu\{\lambda\}=\prod_{k=1}^\infty \bar g_{k,R_k(\lambda)}\,.
\end{equation}
Introduce normalized coefficients $g_{k,j}=\bar g_{k,j}/\bar g_{k,0}$ 
(division by $\bar g_{k,0}$ is possible since $\prod_k \bar g_{k,0}=\bar a_0>0$ 
by definition) and consider functions $f_k(x)=\sum_{j=0}^\infty g_{k,j}x^j$; 
these are analytic functions at least in the unit disk.  For a 
positive parameter $x$ let us define
a family of measures $\mu_x$ on the set of all integer partitions $\CP$ by
\[
\mu_{x}\{\lambda\}=\prod_{k=1}^\infty\frac{g_{k,R_k(\lambda)} x^{kR_k(\lambda)}}
  {f_k(x^k)}
=\frac{x^{N(\lambda)}}{F(x)}\prod_{k=1}^\infty g_{k,R_k(\lambda)}
\]
where 
\begin{equation}\label{eq:Fasprod}
F(x)=\prod_{k=1}^\infty f_k(x^k)\,.
\end{equation}
Inequalities $1\le f_k(x^k)\le 1/\bar g_{k,0}$ valid for $0\le x\le 1$ 
(together with $\bar g_{k,0}>0$ stated above) ensure that the products 
above converge at least for these $x$. Moreover,
summation over all $\lambda\in\CP$ yields that $\mu_x$ are probability measures:
\[
\mu_x\CP = \sum_{\lambda\in\CP}\mu_x\{\lambda\}
=\sum_{(r_1,r_2,\dots)}\prod_{k=1}^\infty\frac{g_{k,r_k} x^{kr_k}} {f_k(x^k)}
=\prod_{k=1}^\infty\sum_{r_k=0}^\infty\frac{g_{k,r_k} x^{kr_k}} {f_k(x^k)} = 1
\]
where the middle sum is taken over all sequences $(r_1,r_2,\dots)$ 
of nonnegative integers with finitely many nonzero terms (such sequences are
in one-to-one correspondence with partitions via $r_k=R_k(\lambda)$).

Equation \eqref{eq:defmu1} and definition of $\bar\mu$ imply that for any $n$
if $\lambda\in\CP(n)$ then
\[
\mu_x\{\lambda\}=\frac{x^n}{\bar a_0F(x)}\bar\mu\{\lambda\}
=\frac{\bar a_nx^n}{\bar a_0 F(x)}\mu^{(n)}\{\lambda\}\,.
\]
Thus measures $\mu_x$ are convex combinations of $\mu^{(n)}$.  Introducing
$a_n=\bar a_n/\bar a_0$ makes it possible to write down an expression for 
measures of partition $\lambda \in\CP(n)$ as
\begin{equation}\label{eq:defmun}
\mu^{(n)}\{\lambda\}=\frac{F(x)}{a_nx^n}\mu_x\{\lambda\}
=\frac{1}{a_n}\prod_{k=1}^\infty g_{k,R_k(\lambda)}
\end{equation}
and the Taylor decomposition of $F(\cdot)$ as
\[
F(x)=\sum_{n=0}^\infty a_nx^n\,.
\]
Since $a_n>0$, this is the analytic function in the unit disk, however the 
actual radius of convergence $\rad$ can be greater (and even infinite).


Function $F$ plays a r\^ole of normalization factor, so a man
with background in statistical mechanics would call it a \textit{partition
function}. We utilize this terminology and extend the analogy with statistical
mechanics by using terms \textit{grand canonical ensemble} of partitions for
the set $\CP$ equipped with measure $\mu_x$ and \textit{small canonical
ensemble} for the pair $(\CP(n),\mu^{(n)})$.  Further discussion of these 
analogy and terms can be found in~\cite{Sinai,V-FAA}.

Partition function $F$ is also closely related to the probabilistic notion of
a probability generating function.  Moments of $N$ can be expressed in
terms of $F$ as 
\begin{equation}\label{eq:exNm}
\BE_x N^m = \frac{1}{F(x)}\left(x\frac{d}{dx}\right)^mF(x)\,,
\qquad m=0,1,2,\dots,
\end{equation}
where $\BE_x$ is an expectation operator with respect to measure $\mu_x$.
Similar formula expresses moments of $R_k$ in terms of $f_k$:
\[
\BE_x R_k^{\,m}=\left.\frac{1}{f_k(z)}\left(z\frac{d}{dz}\right)^mf_k(z)\right|_{z=x^k}\,.
\]
Since $N=\sum kR_k$, mean and variance of $N$ can be also easily expressed
in terms of functions~$f_k$ and get a particularly simple form in terms
of its logarithmic derivative $h_k(z)=f_k'(z)/f_k(z)$:
\begin{align} \label{eq:EN-sum}
\BE_x N&{}=\sum_{k=1}^\infty \frac{k x^k f'_k(x^k)}{f_k(x^k)}=
\sum_{k=1}^\infty k x^k h_k(x^k)\,,\\ \notag
\Var_x N&{}=\sum_{k=1}^\infty 
   \frac{k^2\bigl(x^k\left(f_k'(x^k)+x^kf_k''(x^k)\right)f_k(x^k)
         -x^{2k}\left(f_k'(x^k)\right)^2\bigr)}{\left(f_k(x^k)\right)^2}\\
\label{eq:varN-sum}
&{}=\sum_{k=1}^\infty k^2 \bigl(x^k h_k(x^k)+x^{2k}h_k'(x^k)\bigr)\,.
\end{align}

Note that $F$ itself does not define 
measures $\mu_x$ or $\mu^{(n)}$ but $F$ along with its 
decomposition~\eqref{eq:Fasprod} does.  However this decomposition is not
unique. Indeed, given $F(\cdot)$ we could have taken 
$\bar a_n=a_nx_0^{\,n}/F(x_0)$ in 
Definition~\ref{def:mult}, for some $x_0\in(0,\rad)$, 
and constructed a new function $\hat F(\cdot)$ in the similar way.  
However it would satisfy $\hat F(x)=F(xx_0)$  and 
 $\hat f_k(x)=\hat f_k(xx_0^{\,k})$,
as can be easily checked. Up to this change of variable $F$ and its
 decomposition is uniquely defined.

\subsection*{The case considered in this note}

Although multiplicativity is a rather restrictive requirement on 
measures $\mu^{(n)}$ the range of multiplicative measures is quite big.
In this note we consider only measures $\mu^{(n)}$ such that
after some appropriate change of variables described in the previous paragraph
 $f_k(x)=f(x)^{b_k}$ for some function $f(\cdot)$ and sequence of nonnegative
numbers $\{b_k\}$, i.\;e.
\begin{equation}\label{eq:ourF}
F(x) = \prod _{k=1}^\infty f(x^k)^{b_k}\,,\qquad b_1=1.
\end{equation}
The choice $b_1=1$ eliminates the possibility of an interplay 
between the sequence $\{b_k\}$ and function $f$: for any
$b>0$ one can replace $f$ by $f^b$ and $\{b_k\}$ by $\{b_k/b\}$ to get the 
same measure $\mu_x$.  This normalization is always possible
since Definition~\ref{def:mult} implies that $b_1>0$ (otherwise $a_1=0$ which
is prohibited by the definition).

The natural requirement on the Taylor coefficients of $f(z)^{b_k}$ to be
positive may imply certain restrictions on $b_k$.
We impose another requirement on the sequence $\{b_k\}$, 
namely we assume that partial sums
\begin{equation}\label{eq:Bk}
B_k=\sum_{j=1}^k b_j = k^\beta \ell(k),\qquad \beta>0,
\end{equation}
where $\ell(\cdot)$ is a regularly varying function in the sense of Karamata,
i.\;e.\ it is measurable and for each fixed $y\in(0,\infty)$ there 
exists $\lim_{x\to\infty} \ell(xy)/\ell(x)=1$, see~\cite{BGT}. 

For certain statements below these assumptions on behavior of $b_k$ are not
enough and additional conditions are required. In order to formulate 
the first of them we 
introduce for a positive real $s$ the set $K_s$ of integers behaving
 similar to the arithmetic progression with the difference $s$.  More formally,
define 
\begin{equation}\label{eq:defKs}
K_s=\{k\in\mathbb{Z}_+:\exists j\text{ such that }|k-sj|<1/2\}\,.
\end{equation}
  Obviously, for
$s\le1$ these sets coincide with $\mathbb{Z}_+$ but for $s>1$ holes in $K_s$
occur.  In some statements we shall need the following regularity assumption
in addition to~\eqref{eq:Bk}:
there exists $\chi\in(0,1)$ such that for any $s\ge 2$
\begin{equation}\label{eq:Bkreg}
\sum_{\substack{j\le k\\j\in K_s}}b_j\le \chi B_k\,.
\end{equation}

If $\beta>2$ assumption \eqref{eq:Bk} is strong enough for all our purposes. 
However for $0<\beta\le2$ we need more detailed asymptotics 
of partial sums $B_k$:
\begin{equation}\label{eq:Bkdetailed}
B_k=\theta k^\beta+O\left(k^{\beta-\zeta}\right),\qquad\qquad k\to\infty
\end{equation}
for some constants $\beta,\theta>0$ and $\zeta>1-\beta/2$.

We also suppose that for some $\rad_1\in(0,\infty]$,  $f(x)$ is
finite for $x\in(0,\rad_1)$ and has a nonremovable singularity at $x=\rad_1$;
the nonnegativity of the  Taylor coefficients implies that $f(x^k)$  
is an analytic function in a disk of radius $\rad_k=\rad_1^{1/k}$ 
($\rad_k=\infty$ if $\rad_1=\infty$).  If the singularity happens 
at $\rad_1\le 1$ we shall often require that it is a pole. 
If $\rad_1$ is finite the change of variables $x\mapsto \rad_1 x$ can make the
radius of convergence of all function $f_k(x^k)$ equal~$1$ however
functions $f_k(\cdot)$ won't be equal after it.
The following simple statement holds. 

\begin{proposition}\label{prop:radius} Let $F$ be defined by \eqref{eq:ourF}
and condition~\eqref{eq:Bk} holds.
If $\rad_1<1$ then $F$ 
is analytic in the disk $|z|<\rad_1$ and has a singularity at $\rad=\rad_1$. 
If $\rad_1\ge1$ {\em(}in particular $\rad_1=\infty${\em)} 
then $F(\cdot)$ is analytic in the unit disk and has a
singularity at~$\rad=1$.
\end{proposition}

\begin{proof}
If $x<\rad_1<1$ then for all $0<y<x$ inequality $f(y)\le 1+(f(x)-1)y$ holds
since function $f$ is convex. Consequently the product \eqref{eq:Fasprod}
evaluated at point $y$ is dominated by the converging product 
$\prod_{k=1}^\infty \bigl(1+(f(x)-1)y^k\bigr)^{b_k}$.  On the other hand,
$F(x)\to\infty$ as $x\to\rad_1$ since so does the first
factor in~\eqref{eq:Fasprod}.

If $x<1\le \rad_1$ the same argument shows the convergence of the infinite 
product evaluated at~$x$, but $F(1)=\prod_k f(1)^{b_k} =\infty$
since $B_k\to\infty$ by~\eqref{eq:Bk}.
\end{proof}

\section{Ergodicity}\label{sec:erg}

Given a partition $\lambda$ of $n$ we consider its Young diagram which can
be defined as a subgraph of the function
\[
\varphi_\lambda(t)=\sum_{k>t}R_k(\lambda),
\qquad t\ge 0.
\]
For a sequence of positive numbers $\alpha^{(n)}$ we consider
its scaled version
\[
\tphi^{(n)}_\lambda(t)=\frac{\alpha^{(n)}}{n}\varphi_\lambda(\alpha^{(n)}t)
=\frac{\alpha^{(n)}}{n}\sum_{k>\alpha^{(n)}t}R_k(\lambda)\,.
\]
Taking $\lambda\in\CP(n)$  at random with probability $\mu^{(n)}\{\lambda\}$
makes these random functions.

\begin{definition}\label{def:erg}
We call a family of measures $\mu^{(n)}$ \textit{ergodic} if
there exists a sequence $\alpha^{(n)}$ and a piecewise continuous function 
$\varphi:\BR_+\to\BR_+$
such that $\int_0^\infty \varphi(t)\,dt=1$ and 
for any finite collection $0< t_1<\dots<t_\ell$ of its continuity points 
values $\tphi^{(n)}_\lambda(t_j)$, $j=1,\dots,\ell$, 
converge to $\varphi(t_j)$ in probability,
that is for any $\eps>0$
\[
\lim_{n\to\infty} \mu^{(n)}\bigl\{\lambda\colon
  \bigl|\tphi^{(n)}_\lambda(t_j)-\varphi(t_j)\bigr|<\eps \text{ for all }
  j=1,\dots,\ell\bigr\}=1\,.
\]
The function $\varphi$ is called a \textit{limit shape} of partitions.
\end{definition}

\begin{remarks}
1.\enspace If a sequence $\alpha^{(n)}$ exists it is essentially unique meaning that 
for $\alpha_1^{(n)}$ and $\alpha_2^{(n)}$ two such sequences 
$\alpha_1^{(n)}/\alpha_2^{(n)}\to c\in(0,\infty)$ and function $\varphi$ 
is appropriately transformed.

2.\enspace If a function $\varphi$ exists it is nonincreasing since all
$\tphi^{(n)}_\lambda$ do not increase.
\end{remarks}

The notion of ergodicity can be also defined in the grand canonical ensemble. 
However it should be done in slightly different way to keep the main
advantage of measures $\mu_x$ that $\varphi_\lambda(t)$ is a sum of
independent random variables. Given a positive function $\alpha_x$ defined for
$x\in (0,\rad)$ define the scaled Young diagram as
\[
\tphi_{x;\lambda}(t)=\frac{\alpha_x}{\BE_x N}\,\varphi_\lambda(\alpha_xt)
=\frac{\alpha_x}{\BE_x N}\sum_{k>\alpha_xt}R_k(\lambda)\,.
\]
Scaling here depends on $x$ so $\int_0^\infty\tphi_{x;\lambda}(t)dt=1$
does \textit{not} hold for all $\lambda$ 
however the mean value of this integral is~$1$.
A family of measures $\mu_x$ is called ergodic if there exist a scaling function
$\alpha_x$ and a limit shape $\varphi$, $\int_0^\infty \varphi(t)\,dt=1$, 
such that for any $\eps>0$ and 
$(t_1,\dots,t_\ell)$ a set of continuity points of $\varphi$
\[
\lim_{x\nearrow\rad}\mu_x\bigl\{\lambda\colon
  \bigl|\tphi_{x;\lambda}(t_j)-\varphi(t_j)\bigr|<\eps \text{ for all }
  j=1,\dots,\ell\bigr\}=1\,.
\]
Ergodicity of $\mu^{(n)}$ and $\mu_x$ are closely related however 
not equivalent. The subject of this note is to establish conditions for
ergodicity of the first family but we shall investigate properties of the 
second one as well.  We start with a simple criterion for the case 
when measures $\mu_x$ can not be ergodic. 

\begin{proposition}\label{prop:nonerg}
If the partition function $F$ has an isolated pole in the point $\rad$ 
then measures $\mu_x$ are not ergodic.
\end{proposition}

\begin{proof}
Suppose that $F$ has an isolated pole of order $m\ge1$ at $\rad$.  Then in some 
neighborhood of $\rad$ it can be decomposed into the Laurent series
\begin{equation}\label{eq:laurentF}
F(x)=\sum_{j=-m}^\infty c_j(x-\rad)^j\,,\qquad\qquad c_{-m}\ne 0.
\end{equation}
Using formula \eqref{eq:exNm} we see that as $x\nearrow\rad$ 
\begin{gather} \label{eq:laurentexN}
\BE_x N= \frac{-x mc_{-m}(x-\rad)^{-m-1}+\dots}{c_{-m}(x-\rad)^{-m}+\dots}
\sim\frac{m\rad}{\rad-x} ,
\\ \notag
\BE_x N^2=\frac{m(m+1)xc_{-m}(x-\rad)^{-m-2}+\dots}{c_{-m}(x-\rad)^{-m}+\dots}
\sim\frac{m(m+1)\rad^2}{(\rad-x)^2}
\end{gather}
where dots denote lower order terms.
Hence the variance of $N/\BE_xN$ is bounded away from zero.  Consequently 
there exists $\tau>0$ and $c>0$ such that $\mu_x\{\lambda:
|N(\lambda)/\BE_xN-1|>\tau\}>c$ for all $x$ close to $\rad$. 
  Moreover, since the mean of $N/\BE_xN$ is one,
the one-sided inequality should also take place with positive probability:
$\mu_x\{\lambda: N(\lambda)/\BE_xN<1-\tau\}>c$.

Suppose that measures $\mu_x$ are ergodic with scaling $a_x$ and limit shape
$\varphi$.  Recall that $\varphi$ is a weakly
decreasing piecewise continuous function with unit integral. 
Take the $\varphi$-continuity point $\eps>0$ small enough so that 
\[
\int_0^\eps (\varphi(t)-\varphi(\eps))dt+\int_0^\infty \min\{\eps,\varphi(t)\}dt
<\tau/3.
\]
Geometrically it means that the area of the limit shape $\varphi$ lying lower
than $\eps$ or higher than $\varphi(\eps)$ is less than $\tau/3$.
Let $T=\inf\{t:\varphi(t)<\eps\}$.
Given $\eps$ and $\delta\in(0,\eps)$ define 
a finite collection of points recursively 
by the following procedure:
take $t_0=\eps$ and let $t_i=\inf\{t:\varphi(t)<\varphi(t_{i-1}+0)-\delta\}$
until on some step $t_d\ge T$. (Here $\varphi(t+0)$ is the right limit at $t$.)  
The procedure ends in final number of steps 
since $\varphi(t_{i}+0)\le \varphi(t_{i-1})-\delta$. Define now a function
$\varphi^*:[0,T]\to \BR$  to be equal $\varphi(\eps)-2\delta$ on $[0,t_1]$
and for all $i=2,\dots d$ let $\varphi^*(t)=\varphi(t_{i-1}+0)-2\delta$ on 
$(t_{i-1},t_i]$.  Thus $\varphi^*$ is a piecewise constant function with
discontinuities at $\{t_i\}$ and by construction it satisfies $\varphi^*(t)\le 
\varphi(t)-\delta$ for all $t\in[\eps,T]$.  Consequently ergodicity implies that
$\mu_x\{\lambda:\tphi_{x;\lambda}(t)>\varphi^*(t),t\in[\eps,T]\}\to 1$ as $x\to\rad$.  For all such $\lambda$ 
\[
N(\lambda)=\sum_{k=1}^\infty kR_k(\lambda)
\ge \sum_{k=\intfloor{\eps\alpha_x}}^{\intfloor{T\alpha_x}}kR_k(\lambda)
\ge \BE_xN\int_0^T(\varphi^*(t)-\eps)dt.
\]
Taking $\delta>0$ and $\eps>0$ small enough the last integral can be made
greater than $1-2\tau/3$, thus providing the contradiction which finishes
the proof.
\end{proof}

\begin{remark}
Note that the proof does not use the specific form \eqref{eq:ourF} of 
decomposition~\eqref{eq:Fasprod} and thus the result 
holds for any multiplicative measure. 
\end{remark}

\section{Ergodicity in the grand canonical ensemble}\label{sec:erglce}

Independence of $R_k$ in the grand canonical ensemble allows establishing
sufficient conditions for ergodicity in the grand canonical ensemble.
We start with finding asymptotics of the mean value of $N$ with respect to
measure $\mu_x$.

\begin{lemma}\label{lem:ExN}
Let measure $\mu_x$ be defined by decomposition~\eqref{eq:ourF}
and $b_k$ satisfy~\eqref{eq:Bk}.
If $\rad_1<1$ then $\BE_x N\sim \rad_1f'(x)/f(x)$ as $x\nearrow \rad_1$. 
If $\rad_1\ge 1$ $($if $\rad_1=1$ in addition $f$ has a pole
in $1)$
then as $x\nearrow 1$
\begin{equation}\label{eq:ExN}
\BE_xN\sim \Omega\frac{\ell(1/(1-x))}{(1-x)^{\beta+1}}\,,
\qquad\Omega=\int_0^1 
   \left(|\log u|^{\beta+1}(h(u)+uh'(u))-|\log u|^\beta h(u)\right) du\,
\end{equation}
where $h(u)=f'(u)/f(u)$ is the logarithmic derivative of\/ $f$,  
and $\ell$, $\beta$ are defined in~\eqref{eq:Bk}.
\end{lemma}

\begin{proof}
If $F$ satisfies~\eqref{eq:ourF} then $h_k(u)=b_kh(u)$ in~\eqref{eq:EN-sum} 
hence it can be rewritten as
\[
\BE_xN=\sum_{k=1}^\infty kb_kx^kh(x^k)\,.
\]
If $\rad_1<1$ then the only summand above which goes to infinity as 
$x\to\rad_1$ is the first one, and the sum of all remaining summands is
dominated by the convergent series 
$h(\rad_1^2)\sum_{k=2}^\infty kb_k \rad_1^k$, so
the statement holds. 

Suppose $\rad_1\ge1$. Taking partial sum and using summation by parts yields
\begin{align*}
\sum_{k=1}^m kb_kx^kh(x^k) &{}= 
(m+1)B_mx^{m+1}h(x^{m+1})
+\sum_{k=1}^m B_k\left(kx^kh(x^k)
   -(k+1)x^{k+1}h(x^{k+1})
 \right)\,.
\end{align*}
The first summand vanishes as $m\to\infty$ for fixed $x$, so taking
$m$ large enough it can be hold bounded.  Since $h$ is an analytic 
function in some disk including points $x^{k+1}$ 
and $x^k$ the mean value theorem allows to conclude that
\begin{multline*}
kx^kh(x^k)-(k+1)x^{k+1}h(x^{k+1})
=kx^kh(x^k)-(k+1)x^{k+1}\left(h(x^k)+h'(x^\varkappa)(x^{k+1}-x^k)\right)\\
=-x^{k+1}h(x^k)+k(1-x)x^kh(x^k)+(1-x)(k+1)x^{2k+1}h'(x^\varkappa)
\end{multline*}
where $\varkappa\in[k,k+1]$.
Combining the above formulas we obtain
\begin{align*}
\sum_{k=1}^m kb_kx^kh(x^k)
&{}=(m+1)B_mx^{m+1}h(x^{m+1})-\sum_{k=1}^m B_kx^{k+1}h(x^k)
\\
&\qquad+(1-x)\smash[t]{\sum_{k=1}^m }
    B_k\left(kh(x^k)x^k+(k+1)h'(x^\varkappa)x^{2k+1}\right)\,.
\end{align*}
Take $T>0$ and put $m=[T/(1-x)]$.
Since $B_{[t/(1-x)]}/B_{[1/(1-x)]}\to t^\beta$ uniformly in $t\in(0,T]$ (see
\cite[Th.~1.5.2]{BGT}), after multiplication by 
$(1-x)^{1+\beta}/\ell(1/(1-x))$ the above sums become correspondingly
the Riemann sums for the convergent integrals 
$\int_0^T t^\beta h(e^{-t})e^{-t}dt$ and
$\int_0^T t^{\beta+1}\left(h(e^{-t})e^{-t}+h'(e^{-t})e^{-2t}\right)dt$.\footnote{The 
integrals are convergent even if $\rad_1=1$: since $f$ has a pole at~1,
$h(e^{-t})$ has a simple pole and $h'(e^{-t})$ has a pole of order 2 at~0, 
so multiplication
by $t^\beta$ and $t^{\beta+1}$ kills both singularities.  To be completely 
rigorous, one should take integral from $1/T$ to $T$ in order to speak about 
the Riemann sum, and then use uniformness in $T$ to exchange limits.}
Letting $T\to\infty$ and changing variable $u=e^{-t}$  finish the proof.
\end{proof}

Proposition~\ref{prop:nonerg} shows that ergodicity in the grand canonical
ensemble can not take place for the case $\rad_1<1$ if $f$ has a pole
in this point.  Indeed,  $\sum_{k\ge2} \BE_xR_k$ is bounded 
as $x\nearrow \rad$ and $\BE_xN\sim\BE_xR_1$
is not, so one can take scaling $\alpha_x\equiv1$ to get a limit
of scaled Young diagrams $\varphi_x(t)=1_{[0,1)}(t)R$ with $R$ a (nondegenerate)
limit of $R_1/\BE_x R_1$.  Thus in this case ``almost all'' partitions 
consist mostly of ones, and all larger parts constitute a vanishing fraction of
the whole sum.  Further questions can be asked about the distribution of larger 
parts etc.\ however they are beyond the scope of this note.

\smallskip

Suppose $\rad_1\ge1$.
The change of variables $k\leftrightarrow t/(1-x)$ made implicitly in the
proof of Lemma~\ref{lem:ExN} suggests the choice of the
 scaling function $\alpha_x=1/(1-x)$.
The following lemma states that the mean is not degenerate with this scaling.

\begin{lemma}\label{lem:exphit} 
Suppose that $\rad_1\ge1$ and for $\rad_1=1$ assume additionally
that the singularity of  $f$ in~1 is an isolated pole.
For the scaling function $\alpha_x=1/(1-x)$ the mean value of a scaled 
random Young diagram at point $t>0$ is
\begin{equation}\label{eq:exphit}
\varphi(t):=\lim_{x\nearrow1}\BE_x \tphi_x(t)= \frac{1}{\Omega}
\left(\int_0^{e^{-t}}\left(h(u)+uh'(u)\right)|\log u|^\beta du
    -t^\beta h(e^{-t})e^{-t}\right)\,
\end{equation}
where $h(u)=f'(u)/f(u)$ is the logarithmic derivative of $f$, $\Omega$ is 
defined in~\eqref{eq:ExN} and $\ell$, $\beta$ are defined in~\eqref{eq:Bk}.
If $\rad_1=1$ and $\beta\in(0,1]$ then $\varphi(0)=\infty$, otherwise 
it is finite and the convergence takes place also for $t=0$.
\end{lemma}
\begin{proof}
The proof is similar to that of Lemma~\ref{lem:ExN} so we present only a sketch.
The partial sum is
\begin{equation*}
\sum_{k=m_1}^{m_2}\BE_x R_k=B_{m_2}x^{m_2+1}h(x^{m_2+1})
-B_{m_1-1}x^{m_1}h(x^{m_1})
+\sum_{k=m_1}^{m_2}B_k\left(x^kh(x^k)-x^{k+1}h(x^{k+1})\right)\,.
\end{equation*}
Expression in brackets under the summation sign can be represented as
\[
x^kh(x^k)-x^{k+1}h(x^{k+1})=(1-x)(h(x^k)x^k+h'(x^k)x^{2k+1})
  +\tfrac{1}{2}(1-x)^2h''(x^\varkappa)x^{3k+1}
\]
for $\varkappa\in[k,k+1]$,
so taking $m_1=t/(1-x)$, $m_2=T/(1-x)$ and representing sums by integrals
yields 
\begin{align*}
\sum_{k=m_1}^{m_2}\BE_xR_k\sim \frac{\ell(1/(1-x))}{(1-x)^\beta}\biggl({}&{}
T^\beta e^{-T}h(e^{-T})-t^\beta e^{-t}h(e^{-t})\\
&\qquad+\smash[t]{\int_t^T}
  v^\beta\left(h(e^{-v})+h'(e^{-v})e^{-v}\right)e^{-v}dv\smash[t]{\biggr)}
\end{align*}
for $t>0$. If $\beta>1$ then the integral converges also for
$t=0$. To finish the proof it remains to change variable,
divide by the asymptotic expression~\eqref{eq:ExN} for $\BE_xN$  
and let $T\to\infty$.
\end{proof}

The function $\varphi$ defined by~\eqref{eq:exphit} is a natural candidate 
for the limit shape. To show that the
definition of ergodicity really holds we give a bound for probability of
deviation at a fixed point.

\begin{lemma}\label{lem:muxdev}
Suppose that $\rad_1\ge1$, $f$ has a pole in~1 if $\rad_1=1$, 
and fix $t,\eps>0$.
Then, for $x$ close to~1
\begin{equation*}
\mu_x\bigl\{\lambda\colon
  \bigl|\tphi_{x;\lambda}(t)-\varphi(t)\bigr|>\eps\bigr\}\le e^{-(1-x)^{-\beta/2}}\,.
\end{equation*}
\end{lemma}

\begin{proof}
For fixed $t>0$ and $k>t/(1-x)$ there exist exponential moments
$\BE_xe^{uR_k}=f(x^ke^u)^{b_k}/f(x^k)^{b_k}$ at least for 
$u\in[0,t+\log\rad_1)$.  Moreover, convexity of $f$ implies that 
exponential moments of $\sum_{k>t/(1-x)}R_k$  also exist for such $u$:
argument follows the lines of the proof of Proposition~\ref{prop:radius}.
Hence for fixed $\eps>0$, $u\in(0,t+\log\rad_1)$ and $x$ close enough to~1
\begin{multline}\label{eq:muxdevt1}
\mu_x\bigl\{\lambda\colon\tphi_{x;\lambda}(t)-\varphi(t)\ge\eps\bigr\}
=\mu_x\bigl\{\lambda\colon e^{u\tphi_{x;\lambda}(t)}\ge 
          e^{u(\varphi(t)+\eps)}\bigr\}
\le \frac{\BE_xe^{u\tphi_x(t)}}{e^{u(\varphi(t)+\eps)}}\\
= \frac{\BE_xe^{u(\tphi_x(t)-\BE_x\tphi_x(t))}}
         {e^{u(\varphi(t)-\BE_x\tphi_x(t)+\eps)}}
=e^{-u\eps/2}\prod_{k>t/(1-x)}\BE_x 
\exp\left(\frac{u(1-x)^\beta(1+o(1))}{\Omega\ell(1/(1-x))}(R_k-\BE_xR_k)\right)\\
\le e^{-u\eps/2}\prod_{k>t/(1-x)}\BE_x 
\exp\left(\frac{2u(1-x)^\beta}{\Omega\ell(1/(1-x))}(R_k-\BE_xR_k)\right)
\end{multline}
where we have used Markov's inequality, Lemmas~\ref{lem:ExN} 
and~\ref{lem:exphit} and independence of~$R_k$.  Denote for short 
$\delta=\delta(x)=\frac{2(1-x)^\beta}{\Omega\ell(1/(1-x))}$, note that $\delta(x)\to 0$
as $x\nearrow1$. Each factor in the right-hand side of~\eqref{eq:muxdevt1}
is defined at least for $u<(t+\log\rad_1)\delta^{-1}$, and the product
converges. Since
\[
\BE_x \exp\left(\frac{2u(1-x)^\beta}{\Omega\ell(1/(1-x))}(R_k-\BE_xR_k)\right)
=\exp b_k\left(\log\frac{f(x^ke^{u\delta(x)})}{f(x^k)}
    -u\delta(x)\frac{x^kf'(x^k)}{f(x^k)}\right)
\]
the logarithm of $k$'th factor in~\eqref{eq:muxdevt1}
divided by $b_k$ can be bounded as follows:
\begin{multline*}
 \log\frac{f(x^ke^{u\delta})}{f(x^k)}
    -u\delta\frac{x^kf'(x^k)}{f(x^k)}
\le \log\left(1+\frac{f(x^ke^{u\delta})-f(x^k)}{f(x^k)}\right)
  -u\delta\frac{x^kf'(x^k)}{f(x^k)}\\
\le \frac{f(x^ke^{u\delta})-f(x^k)}{f(x^k)}
  -u\delta\frac{x^kf'(x^k)}{f(x^k)}
\le \frac{x^k(e^{u\delta}-1)f'(x^k)}{f(x^k)}
  -u\delta\frac{x^kf'(x^k)}{f(x^k)}
\le (u\delta)^2\frac{x^kf'(x^k)}{f(x^k)}
\end{multline*}
for $u\le1/\delta$
since $f'(x)$ is nondecreasing function and $e^v-v-1\le v^2$ for $v\in[0,1]$.
Hence continuing \eqref{eq:muxdevt1} we obtain
\[
\mu_x\bigl\{\lambda\colon\tphi_{x;\lambda}(t)-\varphi(t)\ge\eps\bigr\}
\le \exp\left((u\delta)^2\left(\sum_{k>t/(1-x)}b_k\frac{x^kf'(x^k)}{f(x^k)}\right)
-u\eps/2\right)\,.
\]
We have already found the asymptotics of the sum above in the proof of
Lemma~\ref{lem:exphit}: it is asymptotically equivalent to 
$\frac{\ell(1/(1-x))}{(1-x)^\beta}\Omega\varphi(t)$.  Hence the upper bound
takes form $\exp(c_1\delta(x)u^2-u\eps/2)$ for some $c_1>0$. Now we can choose
$u$ such that it provides the best estimate. It is achieved at
$u=\eps/(4c_1\delta(x))$ which gives an upper bound $\exp(-\eps^2/(16 c_1\delta(x)))$ and at least for small $\eps$ all the calculations
 above are valid.  

The same upper bound for $\mu_x\left\{\lambda\colon\tphi_{x;\lambda}(t)
-\varphi(t)\le-\eps\right\}$ is obtained exactly in the same way.
An observation that 
$2\exp(-\eps^2/(16 c_1\delta(x)))<\exp(-(1-x)^{\beta/2})$ for $x$ close enough
to~1 finishes the proof..
\end{proof}

We combine the results about ergodicity in the grand canonical ensemble in
the next statement.

\begin{thm}\label{thm:lceerg}
Measures $\mu_x$ defined by~\eqref{eq:ourF} with $b_k$ satisfying \eqref{eq:Bk}
are ergodic if either the radius of convergence $\rad_1$ 
of function $f$ is greater
than~1 or if it is equal to~1 and $f$ has a pole at~1.
The possible choice of scaling function in this case is
$\alpha_x=1/(1-x)$ which leads to the limit shape $\varphi$ 
defined by~\eqref{eq:exphit}.

If $\rad_1<1$ and $f$ has a pole at $\rad_1$
 then ergodicity does not hold.  
\end{thm}

\begin{proof}
Lemma~\ref{lem:muxdev} gives the exponential upper bound for the probability
of deviation greater than $\eps>0$ of $\tphi_{x;\lambda}(t)$ from $\varphi(t)$.  
Consequently the
probability of deviation greater than $\eps$ in
finite number of points still decays exponentially as $x\to1$. 

The last statement follows from Proposition~\ref{prop:nonerg}.
\end{proof}

\begin{remark}
If $\rad_1<1$ and $f$ has an essential singularity at~$\rad_1$
then ergodic case is possible: take, say, $f(x)=e^{1/(1-2x)}$ and $b_k=1$.
\end{remark}

\section{Ergodicity in the small canonical ensemble}\label{sec:ergsce}

In order to approximate measures $\mu^{(n)}$ by measures $\mu_x$ we want to 
choose $x$ depending on $n$ to maximize $\mu_x\CP(n) =a_nx^n/F(x)$.  
Differentiation with respect to $x$ shows that it is achieved
 at $x=x_n$, a solution of equation
\begin{equation}\label{eq:defxn}
n=\BE_{x_n}N=\frac{x_nF'(x_n)}{F(x_n)}\,.
\end{equation}
Note that this solution always exists and is unique since $\BE_xN$
strictly increase in~$x$. 
Lemma~\ref{lem:ExN} and \cite[Prop.~1.5.15]{BGT} shows that for
$\rad_1\ge 1$ in the settings of Lemma~\ref{lem:ExN} there
exists a slowly varying function~$\ell_1$ such that 
\begin{equation}\label{eq:taun}
\tau_n:=1-x_n=\frac{\ell_1(n)}{n^{1/(\beta+1)}}\,.
\end{equation}
In the most simple case when $\ell(k)\equiv1$ this simplifies to 
$\tau_n\sim (\Omega/n)^{1/(\beta+1)}$, in the general case 
$\ell_1$ is connected to the de~Bruijn conjugate of $\ell$, see~\cite{BGT}. 

Theoretically it could happen that the maximal probability 
$\mu_{x_n}\CP(n)$ is still small
to guarantee that the conditional measures $\mu^{(n)}=\mu_{x_n}\big|_{\CP(n)}$ 
exhibit the same behavior as unconditional ones. To eliminate this possibility
we use the local limit theorem for $N$. This result is much stronger than 
needed however it is interesting in itself.  

Note that equation \eqref{eq:varN-sum} implies that $\Var_{x}N\to\infty$ 
as $x\nearrow 1$.  
Moreover this equation and arguments similar to the proof 
of Lemma~\ref{lem:ExN} yields that 
\begin{equation}\label{eq:varxN}
\Var_{x}N=\smash[b]{\sum_{k=1}^\infty 
k^2b_k(x^{k}h(x^{k})+x^{2k}h'(x^{k}))\\
\sim \frac{\ell(1/(1-x))}{(1-x)^{\beta+2}}\sigma^2}
\end{equation}
where
\[
\sigma^2=\int_0^1\left(2|\log u|^{\beta+1}\left(h(u)+uh'(u)\right)
  -|\log u|^{\beta+2}\left(h(u)+3uh'(u)+u^2h''(u)\right)\right)du.
\]
The integral is convergent even if $\rad_1=1$ because in this case
$h'(u)$ has a pole of order~2 and $h''(u)$ has a pole of order 3 at $u=1$
and both $\int^1 |\log u|^{\beta+1}h'(u)du$ and 
$\int^1 |\log u|^{\beta+2}h''(u)du$ converge.

\begin{lemma}[Local limit theorem]\label{lem:lltN}
Let measures $\mu_x$ be defined by decomposition \eqref{eq:ourF} 
where either the radius of convergence $\rad_1>1$ 
or $\rad_1=1$ and additionally $f$ has a pole at~$1$. Let the
sequence $b_k$ satisfy both conditions~\eqref{eq:Bk} and~\eqref{eq:Bkreg}
and additionally either $\beta>2$ or $0<\beta\le2$ and 
\eqref{eq:Bkdetailed} holds for some $\zeta>1-\beta/2$.

Suppose that integer-valued function $m(x)$ grows so that 
\[
\frac{m(x)-\BE_xN}{\sqrt{\Var_xN}}\to u\,,\qquad x\nearrow 1\,,
\]
for some constant~$u$. Then
\[
\sqrt{\Var_xN}\,\mu_x\CP(m(x))-\frac{1}{2\pi}e^{-\frac{u^2}2}\to 0\,,
\qquad x\nearrow 1\,.
\]
\end{lemma}

\begin{proof}
We start with the Cauchy formula for $a_m$ where we take the circle of radius
$x$ as the integration path:
\[
a_m=\frac{1}{2\pi}\int_{-\pi}^\pi F(xe^{\ii t})x^{-m}e^{-\ii mt}dt\,.
\]
Using product representation \eqref{eq:ourF} and expressing $\BE_xN$ in terms
of $f$ by \eqref{eq:EN-sum} gives
\begin{equation}\label{eq:lwrbnd1}
\begin{split}
\mu_{x}\CP(m)&{}=\frac{a_mx^{m}}{F(x)}
=\frac{1}{2\pi}\int_{-\pi}^\pi e^{-\ii mt}\prod_{k=1}^\infty 
\frac{f(x^{k}e^{\ii kt})^{b_k}}{f(x^{k})^{b_k}}\,dt\\
&{}=\frac{1}{2\pi}\int_{-\pi}^\pi \exp\left(\ii(\BE_xN-m)t+\sum_{k=1}^\infty 
b_k\left(\log\frac{f(x^{k}e^{\ii kt})}{f(x^{k})}
-\ii kt\frac{x^{k}f'(x^{k})}{f(x^{k})}\right)\right)dt\,.
\end{split}
\end{equation}
The function under the integral sign 
in~\eqref{eq:lwrbnd1} sends values of~$t$ with opposite signs 
to complex conjugates. Thus
taking the real part does not change the value of the integral and
changing the integration interval to $[0,\pi]$ halves its value.  
So we can write
\[
\mu_{x}\CP(m)=I_1+I_2+I_3+I_4\ge I_1-|I_2|-|I_3|-|I_4|
\]
where for some 
$0=\delta_0(x)\le\delta_1(x)\le\delta_2(x)\le\delta_3(x)\le\delta_4(x)
=\pi$ we denote for $j=1,2,3,4$
\begin{equation}\label{eq:lwrbnd1a}
I_j=\frac{1}{\pi}\int_{\delta_{j-1}(x)}^{\delta_j(x)}
\Re\exp\left(\ii(\BE_xN-m)t+\sum_{k=1}^\infty 
b_k\left(\log\frac{f(x^{k}e^{\ii kt})}{f(x^{k})}
-\ii kt\frac{x^{k}f'(x^{k})}{f(x^{k})}\right)\right)dt \,.
\end{equation}

We define $\delta_i(x)$ as 
\[
\delta_1(x)=(1-x)^{1+\beta/2-\alpha_1},\qquad
\delta_2(x)=1-x,\qquad
\delta_3(x)=(1-x)^{\alpha_3}
\]
where the obvious inequalities which provide the right order for $\delta_i$ are
$0<\alpha_1\le\beta/2$ and $0<\alpha_3\le 1$. Exact values of $\alpha_1$ and
$\alpha_3$ will be chosen later.

\smallskip
First we show that $\alpha_1$ can be chosen so that  $I_1$ gives the
main contribution and then put an upper bound on other integrals.  
Since $0<t<\delta_1(x)\searrow 0$ as $x\nearrow 1$, in order to estimate the sum 
in the exponent in equation~\eqref{eq:lwrbnd1a} for $j=1$
we are going to find $k_0=k_0(x)$ such that for
$k\le k_0$ each summand can be efficiently estimated using the Taylor formula
and the sum over $k>k_0$ is small.  
To this end we take $k_0(x)=\intfloor{(1-x)^{-1-\eps_1}}$ for some
$\eps_1>0$ which exact value will be specified later.  Then for all
$k\ge k_0$ 
\begin{equation}\label{eq:kgek0}
x^k
=\left(1-(1-x)\right)^{(1-x)^{-1}k(1-x)}
\le e^{-k(1-x)}\le e^{-(1-x)^{-\eps_1}}\searrow 0
\end{equation}
in view of inequality $(1-y)^{1/y}\le e^{-1}$ valid for $y\in(0,1]$.
On the other hand, if $k\le k_0$ and $0<t<\delta_1(x)$ then 
once $\alpha_1+\eps_1<\beta/2$
\begin{equation}\label{eq:klek0}
kt\le (1-x)^{\beta/2-\alpha_1-\eps_1}\searrow 0\,.
\end{equation}

If $\alpha_1+\eps_1<\beta/2$ then expression~\eqref{eq:varN-sum} allows 
to estimate the sum
in the exponent of \eqref{eq:lwrbnd1a} as follows:
\begin{multline}\label{eq:sumunderxpeq}
\sum_{k=1}^\infty b_k
\left( 
\log\frac{f(x^{k}e^{\ii kt})}{f(x^{k})}
-\ii kt\frac{x^{k}f'(x^{k})}{f(x^{k})}\right)
+\frac{t^2}{2}\Var_xN\\
{}=\sum_{k=1}^\infty b_k\left( 
\log\frac{f(x_n^{\,k}e^{\ii kt})}{f(x_n^{\,k})}
-\ii ktx^{k}h(x^{k})
+\frac{k^2t^2}{2}\left(x^{k}h(x^{k})+x^{2k}h'(x^{k})\right)\right)\,.
\end{multline}
For $x$ close enough to~1 inequalities~\eqref{eq:kgek0} 
and~\eqref{eq:klek0} guarantee that $\log f(x^ke^{\ii kt})$ is analytic 
hence the following Taylor expansions are valid: for $k\le k_0(x)$ and
$0\le t\le\delta_1(x)$
\begin{align*}
\log{}\frac{f(x^ke^{\ii kt})}{f(x^k)}
&{}=\ii ktx^kh(x^k)-\frac{k^2t^2}{2}\left(x^kh(x^k)+x^{2k}h'(x^k)\right)
-\frac{\ii k^3}{2}\int_0^th_1(x^ke^{\ii ks})(t-s)^2ds\\
\intertext{and for $k>k_0$ and for any real $t$}
\log f(x^ke^{\ii kt})
&{}=\log f(x^k)+x^k(e^{\ii kt}-1)h(x^k)
+\smash[t]{\int_{x^k}^{x^ke^{\ii kt}}}h'(z)(x^ke^{\ii kt}-z)dz
\end{align*}
where $h_1(z)=zh(z)+3z^2h'(z)+z^3h''(z)$.  Thus 
\begin{equation}\label{eq:sumunderxple}
\begin{split}
\sum_{k=1}^\infty b_k
&\left( 
\log\frac{f(x^{k}e^{\ii kt})}{f(x^{k})}
-\ii kt\frac{x^{k}f'(x^{k})}{f(x^{k})}\right)
+\frac{t^2}{2}\Var_xN\\
&{}\le \sum_{k=1}^{k_0}\frac{k^3b_k}{2}\int_0^t
\left|h_1(x^ke^{\ii ks})\right|(t-s)^2ds\\
&\qquad{}+\sum_{k=k_0+1}^\infty b_k\left(
\frac{k^2t^2}{2}x^k\left|h(x^k)\right|
+k\int_{0}^{t}\left|h'(x^ke^{\ii ks})\right|x^{2k}
  \left|e^{\ii kt}-e^{\ii ks}\right|\,ds\right)\\
&{}\le c_1\sum_{k=1}^{k_0}k^3b_kt^3\left(
\left|h_1(x^k)\right|+x^k\right)
+c_2\sum_{k=k_0+1}^\infty \left(k^2b_kt^2x^k+k^3b_kt^kx^{2k}\right)\\
&{}\le c_3t^2 \ell(1/(1-x))(1-x)^{-2-\beta/2-\alpha_1}
=O\left(t^2(1-x)^{\beta/2-\alpha_1}\Var_xN\right)
\end{split}
\end{equation}
in view of inequalities $\left|e^{\ii s}-1-\ii s\right|\le s^2/2$ 
valid for all real $s$,
$\left|h_1(x^{k}e^{\ii ks})\right|\le 
c_1(\left|h_1(x^{k})\right|+x^{k})$ valid for $k\le k_0(x)$ and 
$s\in[0,\delta_1(x)]$, and $\left|h(x^k)\right|<c_2$, 
$\left|h(x^ke^{\ii ks})\right|<c_2$ valid for all $k>k_0(x)$ and all real $s$.
(Here and below $c_1,c_2,\dots$ are some positive constants.) The sum from
$0$ to $k_0$ is estimated by an integral as above and inequality 
$t<\delta_1(x)$ is applied; the sum over $k>k_0$ is exponentially small 
by~\eqref{eq:kgek0}.

Inequality \eqref{eq:sumunderxple} allows us to estimate $I_1$ as follows:
\begin{align*}
I_1&{}=\frac{1}{\pi}\int_{0}^{\delta_1(x)}
\Re\exp\left(\ii(\BE_xN-m)t+\sum_{k=1}^\infty 
b_k\left(\log\frac{f(x^{k}e^{\ii kt})}{f(x^{k})}
-\ii kt\frac{x^{k}f'(x^{k})}{f(x^{k})}\right)\right)dt \\
&{}=\frac{1}{\pi}\int_0^{\delta_1(x)}\Re\exp\left(\ii(\BE_xN-m)t
-\frac{t^2\Var_{x}N}{2}\left(1+
   O\left((1-x)^{\beta/2-\alpha_1}\right)
\right)\right)dt\\
&{}\sim\frac{1}{\pi\sqrt{\Var_{x}N}}\int_0^{\infty}
  \Re\exp\left(-\ii u t-\frac{t^2}{2}\right)ds\\
&{}=\frac{1}{\sqrt{2\pi\Var_xN}}\,e^{-\frac{u^2}{2}}.
\end{align*}

\smallskip

For integrals $I_2$, $I_3$ and $I_4$ we are going to find an exponential
bound
\begin{equation}\label{eq:boundI234}
|I_j|\le \exp\left(-(1-x)^{-\beta/2}\right),\qquad j=2,3,4.
\end{equation}
These bounds are based on the same estimate:
\begin{equation}\label{eq:lwrbnd2}
\begin{split}
|I_j|&{}\le \frac{1}{\pi} \int_{\delta_{j-1}(x)}^{\delta_j(x)} \left|
\exp\sum_{k=1}^\infty b_k\log \frac{f(x^{k} e^{\ii kt})}{f(x^{k})}
  -\ii kt\frac{x^{k}f'(x^{k})}{f(x^{k})}
\right|dt\\
&{}\le \frac{1}\pi \int_{\delta_{j-1}(x)}^{\delta_j(x)}
 \exp\left(\sum_{k=1}^\infty b_k
   \Re \log\frac{f(x^{k}e^{\ii kt})}{f(x^{k})}\right)dt\\
&{}= \frac{1}\pi  \int_{\delta_{j-1}(x)}^{\delta_j(x)} 
\exp\left(\sum_{k=1}^\infty b_k
   \log\left|\frac{f(x^{k}e^{\ii kt})}{f(x^{k})}\right|\right)dt\,.
\end{split}
\end{equation}
Since the Taylor coefficients
of $f$ are nonnegative, each summand in the exponent is nonpositive.  Moreover,
since the first Taylor coefficient $g_1>0$ and thus for all $y\in(0,1)$ and
real~$t$
\begin{multline}\label{eq:lwrbnd3}
-\log\left|\frac{f(ye^{\ii s})}{f(y)}\right|
=-\log \left(1-\frac{f(y)-|f(ye^{\ii s})|}{f(y)}\right)\\
\ge \frac{f(y)-|f(ye^{\ii s})|}{f(y)}
=\frac{f(y)^2-|f(ye^{\ii s})|^2}{f(y)(f(y)+|f(ye^{\ii s})|)}
\ge \frac{g_1y(1-\cos s)}{f(y)^2}
\end{multline}
because 
\[
f(y)^2-|f(ye^{\ii s})|^2=2\sum_{j>k\ge0}g_jg_ky^{j+k}(1-\cos((j-k)s))
\ge 2g_0g_1y(1-\cos s)\,.
\]
If $\rad_1>1$ then $f(y)$ is bounded in the left neighborhood of~1;
if $\rad_1=1$ and $f(y)$ has a pole of order $m$ then $f(y)\le c_3(1-y)^{-m}$.
Taking $m=0$ if $\rad_1>1$ 
implies that for some $\tilde x\in(0,1)$ and
all $y\in [\tilde x,1)$ 
\begin{equation}\label{eq:lwrbndat1}
\frac{g_1y}{f(y)^2}\ge c_4(1-y)^{2m}.
\end{equation}
On the other hand, for small $y$ there exists a constant $c_5>0$ such that
\begin{equation}\label{eq:lwrbndat0}
\frac{g_1y}{f(y)^2}\ge c_5y
\end{equation}
It can be shown that this inequality can be extended to the set 
$y\in(0,\tilde x]$ with the same $\tilde x$ as above
(but with smaller $c_5$, possibly), 
however we do not rely on this fact.
 
\smallskip

In order to find constraints on $|I_2|$, $|I_3|$ and $|I_4|$ it suffices to
take just some summands in the sum in the right-hand part of~\eqref{eq:lwrbnd2}
(and replace all other summands by zero).
The right choice differs for these integrals, and we start with  $|I_2|$.

Inequality $1-\cos y\ge 2y^2/\pi^2$ holds for $y\in[-\pi,\pi]$. 
Since $t\le \delta_2(x)$ is small, for $x$ close to~1 a lot of values $kt$ get
into this interval and it suffices to sum only over these $k$.  Namely, we
take $k\le k_1(x)=\intfloor{\eta/(1-x)}$ where 
$\eta=\min\{|\log\tilde x|/(2\log 2),\pi\}$.
For these $k$ and $x$ close to~1, on the one hand
\[
x^{k}=e^{k\log x }\ge e^{-2k(1-x)\log 2}\ge 
e^{-|\log\tilde x|}=\tilde x
\]
since $\log x\ge-2(1-x)\log 2$ for $x\in[1/2,1]$,
and on the other hand $kt\le\pi$ for $t\le \delta_2(x)$.  Hence the bound 
\eqref{eq:lwrbndat1} applies and taking $y=x^{k}$ and $s=kt$ 
in~\eqref{eq:lwrbnd3} yields
\[
-\log\left|\frac{f(x^{k}e^{\ii kt})}{f(x^{k})}\right|
\ge c_4(1-x^{k})^{2m}\,\frac{2(kt)^2}{\pi^2}\,.
\]
It is easier to utilize the Stieltjes integrals instead of summation
by parts in this case and we introduce $B(u)=B_{\intfloor{u}}$ for this purpose.
Integration by parts gives
\begin{equation}\label{eq:lwrbnd4}\begin{split}
\sum_{k=1}^{k_1}b_k(1-x^{k})^{2m}k^2
&{}=\int_{\frac12}^{k_1+\frac12}\left(1-x^{u}\right)^{2m}u^2dB(u)\\
&\hspace{-40pt}=\left.\left(1-x^{u}\right)^{2m}\left(u^2B(u)-2B_1(u)\right)
   \right|_{u=\frac12}^{k_1+\frac12}\\
&-2m|\log x|\int_{\frac12}^{k_1+\frac12}\left(u^2B(u)-2B_1(u)\right)
  x^{u}\left(1-x^{u}\right)^{2m-1}du\\
\end{split}
\end{equation}
where 
\[
B_1(u)=\int_{\frac12}^u vB(v)\,dv \sim \frac{\ell(u)u^{\beta+2}}{\beta+2}
\,,\qquad\qquad u\to\infty\,.
\]
The first summand 
\[
\left.\left(1-x^{u}\right)^{2m}\left(u^2B(u)-2B_1(u)\right)
   \right|_{u=\frac12}^{k_1+\frac12} \sim \frac{\beta}{\beta+2}\left(1-e^{-\eta}\right)^{2m}\ell(k_1)k_1^{\beta+2}
\]
and for $m=0$ equation \eqref{eq:lwrbnd4} gives the desired asymptotics.
For positive $m$ more 
work is needed to show that the difference in the right-hand part 
of~\eqref{eq:lwrbnd4} does not vanish and lower the growth rate.  To this end
we use the following observation:  for any $\eps>0$ and for large 
enough $k_1$ inequality 
$u^2B(u)-2B_1(u)<(1+\eps)\beta\ell(k_1)k_1^{2+\beta}/(2+\beta)$ 
holds for all $u\le k_1$ hence
\begin{multline*}
2m|\log x|\int_{\frac12}^{k_1+\frac12}\left(u^2B(u)-2B_1(u)\right)
  x^{u}\left(1-x^{u}\right)^{2m-1}du\\
=2m|\log x|\left(\int_{\frac12}^{\frac{k_1+1}2}
   +\int_{\frac{k_1+1}2}^{k_1+\frac12}\right)
     \left(u^2B(u)-2B_1(u)\right)
        x^{u}\left(1-x^{u}\right)^{2m-1}du\\
\le (1+\eps)2m|\log x|\ell(k_1)\left(\frac{(k_1+1)^{\beta+2}}{2^{\beta+2}}
       \int_{\frac12}^{\frac{k_1+1}2}
   +(k_1+\tfrac12)^{\beta+2}\int_{\frac{k_1+1}2}^{k_1+\frac12}\right)
        x^{u}\left(1-x^{u}\right)^{2m-1}du\\
\le (1-\eps_1) \frac{\beta}{\beta+2}\left(1-e^{-\eta}\right)^{2m}\ell(k_1)k_1^{\beta+2}
\end{multline*}
for the suitable choice of $\eps_1>0$.  Consequently for $t\le \delta_2(x)$
\[
-\sum_{k=1}^\infty
b_k\log\left|\frac{f(x^{k}e^{\ii kt})}{f(x^{k})}\right|
\ge c_6 t^2\ell(1/(1-x))(1-x)^{-\beta-2}\,.
\]
It follows that for all $t\in[\delta_1(n),\delta_2(n)]$ the exponential
bound~\eqref{eq:boundI234} on $|I_2|$ holds.

\smallskip 
Let us proceed now with $I_4$. 
Suppose that $\delta_3(x)\le t\le\pi$.  For all
$k\ge k_2(x)=\intfloor{|\log\tilde x|/(1-x)+1}$ the inequality 
$x^{k}<\tilde x$
holds and hence bound~\eqref{eq:lwrbndat0} applies.  Thus for any 
$\eps=\eps(n)>0$
\[
-\sum_{k=1}^\infty b_k\log\left|\frac{f(x^{k}e^{\ii kt})}{f(x^{k})}\right|
\ge c_7\eps\sum_{\substack{k=k_2\\1-\cos kt>\eps}}^\infty b_kx^{k}.
\]
The strategy is to add also summands for which $1-\cos kt\le \eps$ and to choose
$\eps$ small enough so that condition \eqref{eq:Bkreg} guarantees that
additional summands do not change the asymptotics too much.  
The sum over all $k\ge k_2$ is greater than
by $c_8\ell(1/(1-x))(1-x)^{-\beta}$.
Take $\eps=\delta_3(x)^2/4$, if $1-\cos kt\le \eps$ then there exists
$j\in\mathbb{Z}$ such that $kt-2\pi j\in [-\arccos(1-\eps),\arccos(1-\eps)]$ 
and consequently $k\in K_{2\pi/t}$ 
(recall definition~\eqref{eq:defKs}).  
Now apply the assumption~\eqref{eq:Bkreg} to see that
\[
c_7\eps\sum_{\substack{k=k_2\\1-\cos kt>\eps}}^\infty b_kx^{k}
\ge c_9(1-\chi)\eps \ell(1/(1-x))(1-x)^{-\beta}\,.
\]
If $\alpha_3<\beta/2$ the right-hand side of the above inequality grows
to $\infty$ providing a proper bound~\eqref{eq:boundI234} for~$|I_4|$.

If $\beta>2$ we can take $\alpha_3=1$ and still have an exponential bound 
for $|I_4|$.  But this choice of $\alpha_3$ implies $\delta_2(n)=\delta_3(n)$
and $I_3=0$. Hence the theorem is proved for the case $\beta>2$.
\smallskip

If $0<\beta\le 2$ we still need a bound for $I_3$ and we obtain it
under additional assumption~\eqref{eq:Bkdetailed} on $\ell$, 
that is $\ell(k)=\theta+O(k^{-\zeta})$, $\zeta>1-\beta/2$.   
Choose $\alpha_3$ such that $1-\zeta<\alpha_3<\beta/2$ and
suppose $t\in[\delta_2(x),\delta_3(x)]$.  In order to
estimate $|I_3|$ we consider
a sum over $k$ for which $1-\cos kt$ is large enough but $x^{k}$ is still
not too small.

To be more precise, let us introduce intervals 
$\mathfrak{I}_j=[m_0(j),m_1(j)]$  
where $m_0(j)=\intfloor{\pi(6j+1)/(3t)}$ and $m_1(j)=\intfloor{\pi(6j+5)/(3t)}$.
Then $k\in\mathfrak{I}_j$ implies $1-\cos kt>1/2$ for any $j$.
If $j\ge j_0=\intfloor{|\log \tilde x|t/(2\pi(1-x))}+1$ then for any
$k\in \mathfrak{I}_j$ one has $x^{k}<\tilde x$ and thus inequality 
\eqref{eq:lwrbndat0} applies with $y=x^{k}$. Take $j_1=2j_0$. Then
\[
-\sum_{k=1}^\infty b_k\log\left|\frac{f(x^{k}e^{\ii kt})}{f(x^{k})}\right|
\ge \sum_{j=j_0}^{j_1-1}\sum_{k\in\mathfrak{I}_j}
2c_5x^{k}(1-\cos kt)
\ge c_5\sum_{j=j_0}^{j_1-1}\left(B_{m_1(j)}-B_{m_0(j)}\right)x^{m_1(j)}.
\]
Detailed asymptotics \eqref{eq:Bkdetailed} for $B_k$ implies that
\begin{align*}
B_{m_1(j)}-B_{m_0(j)}\ge {}&\theta\left(\intfloor{\pi(6j+5)/(3t)}^\beta-\intfloor{\pi(6j+1)/(3t)}^\beta\right)
-c_{10}\frac{j^{\beta-\zeta}}{t^{\beta-\zeta}}\\
\ge {}& c_{11}\frac{j^{\beta-1}}{t^{\beta}}
-c_{10}\frac{j^{\beta-\zeta}}{t^{\beta-\zeta}}
=c_{11}\frac{j^{\beta-1}}{t^\beta}\left(1-\frac{c_{10}}{c_{11}}\,
j^{1-\zeta}{t^\zeta}\right)\,.
\end{align*}
Since $j<j_1\le |\log \tilde x|t/(\pi(1-x))+2$ and $t<(1-x)^{\alpha_3}$ 
the expression in brackets above is bounded from below 
by $1-c_{12}(1-x)^{\alpha_3+\zeta-1}$ and tends to $1$ as $x\nearrow 1$
by the choice of $\alpha_3$.
At the same time $x^{m_1(j)}\ge \tilde x^4$ for $j<j_1$.
Thus 
$B_{m_1(j)}-B_{m_0(j)}\ge c_{13}j^{\beta-1}t^{-\beta}$ and
\[
-\sum_{k=1}^\infty b_k\log\left|\frac{f(x^{k}e^{\ii kt})}{f(x^{k})}\right|
\ge c_{14}j_0^{\beta}t^{-\beta}\ge c_{15}(1-x)^{-\beta}
\]
providing the inequality~\eqref{eq:boundI234}.  This observation finishes the proof. 
\end{proof}

Actually we just need to know how fast $\mu_{x_n}\CP(n)$ goes to zero.
It follows from Lemma~\ref{lem:lltN} that certain negative power of $n$
provides a lower bound.

\begin{corollary}\label{cor:lwrbnd}
In the settings of Lemma~\ref{lem:lltN} for $n$ large enough
\begin{equation}\label{eq:lwrbnd}
\mu_{x_n}\CP(n)\ge n^{-\gamma}
\end{equation}
 for any $\gamma>\frac{\beta+2}{2\beta+2}$
where $x_n$ is the solution of~\eqref{eq:defxn}.
\end{corollary}

\begin{proof}
The claim follows from Lemma~\ref{lem:lltN} by taking $m=n$ and
$x=x_n$.  Indeed, from \eqref{eq:defxn} and \eqref{eq:varxN} we see that
\begin{equation}\label{eq:varxnN}
\Var_{x_n} N = \ell_2(n)n^{-\frac{\beta+2}{\beta+1}}
\end{equation}
where 
$\ell_2(n)=\sigma^2\ell\bigl(n^{1/(\beta+1)}/\ell_1(n)\bigr)/\ell_1(n)^{\beta+2}$
is slowly varying.  Hence $n^{-\gamma}<1/\sqrt{\Var_{x_n}N}$ for large $n$
by~\cite[Prop.~1.5.1]{BGT} and \eqref{eq:lwrbnd} follows.
\end{proof}

\begin{thm}\label{thm:sceerg}
Let measures $\mu^{(n)}$ induce measures in the grand canonical ensemble
such that the decomposition of $F$ in product can be written in 
form~\eqref{eq:ourF} with $b_k$ satisfying \eqref{eq:Bk}.  
Suppose also that for some
$\gamma>0$ inequality~\eqref{eq:lwrbnd} holds 
{\em(}which is true, in particular, 
in assumptions of Lemma~\ref{lem:lltN}{\em)}. 

In these settings if either $\rad_1>1$ or $\rad_1=1$ and $f$ has 
an isolated pole at~$1$ 
then measures $\mu^{(n)}$ are ergodic
with the scaling function 
\[
\alpha^{(n)}=1/(1-x_n)=\frac{n^{1/(\beta+1)}}{\ell_1(n)}
\] 
where $x_n$ is the solution of equation~\eqref{eq:defxn}
and $\ell_1$ is a slowly varying function defined in~\ref{eq:taun}. 
This choice of scaling function leads
to the limit shape $\varphi$ defined by~\eqref{eq:exphit}.
\end{thm}

\begin{proof}
Lemma \ref{lem:muxdev} gives the exponential bound for $\mu_x\bigl\{\lambda:
\bigl|\tphi_{x;\lambda}(t)-\varphi(t)\bigr|>\eps\bigr\}$. Evaluation at 
$x=x_n$ taking \eqref{eq:taun} into account gives
\begin{equation}\label{eq:sceerg1}
\mu_{x_n}\bigl\{\lambda\colon
  \bigl|\tphi_{x_n;\lambda}(t)-\varphi(t)\bigr|>\eps\bigr\}\le 
   e^{-n^{-\beta/(3\beta+3)}}\,.
\end{equation}
Let $\alpha^{(n)}=\alpha_{x_n}$. Then for $\lambda\in\CP(n)$ the scalings
on the grand canonical and small canonical ensembles coincide and
$\tphi^{(n)}_\lambda (t)\equiv \tphi_{x_n;\lambda}(t)$ so
\begin{align*}
\mu^{(n)}\left\{\lambda:\bigl|\tphi^{(n)}_\lambda (t)-\varphi(t)|>\eps\right\}
&{}=\frac{\mu_{x_n}\left(\left\{\lambda:\bigl|\tphi_{x_n;\lambda}(t)-\varphi(t)|>\eps\right\}
\cap \CP(n)\right)}{\mu_{x_n}\CP(n)}\\
&{}\le\frac{\mu_{x_n}\left\{\lambda:\bigl|\tphi_{x_n;\lambda}(t)-\varphi(t)|>\eps\right\}
}{\mu_{x_n}\CP(n)}\,.
\end{align*}
Inequalities \eqref{eq:sceerg1} and~\eqref{eq:lwrbnd} imply that 
this probability tends to~0 as $n\to\infty$.  
Probability of deviations greater than $\eps$ in
finite number of points is bounded by the number of points times the maximal
probability of deviation greater than $\eps$ and also tends to~0, 
proving the ergodicity.
\end{proof}

The case $\rad_1<1$ is nonergodic in the grand canonical ensemble. However
it seems that measures $\mu^{(n)}$ are still ergodic but the limit shape is 
degenerate. 

\begin{conjecture}\label{conj:sceerg}
If $\rad_1<1$ in the settings of Theorem~\ref{thm:sceerg} and $f$ has 
a pole at $\rad_1$ 
then the possible choice of the scaling function could be 
$\alpha^{(n)}\equiv1$ and it leads to the degenerate limit shape
 $\varphi(t)=\mathbf{1}(t\in[0,1])$.
\end{conjecture}

We reinforce this conjecture by the following simple statement.

\begin{proposition}\label{prop:sceerg}
Conjecture~\ref{conj:sceerg} is true if all $b_k>b>0$.
\end{proposition}

\begin{proof}
First of all note that  $\varphi(t)=\mathbf{1}(t\in[0,1])$ would be the limit
shape in the scaling $\alpha^{(n)}=1$ if for any $\eps>0$
$\lim_n\mu^{(n)}\CD_n(\eps)=1$ where 
$\CD_n(\eps)=\bigl\{\lambda\in\CP\colon \varphi_\lambda(t)<n\eps$
for $t\in(1,2)\bigr\}$. Indeed, then the measure of $\lambda\in\CP(n)$ such that
$\tfrac{1}{n}\varphi_\lambda(t)\to0$ 
for all $t>1$ goes to~1 by 
monotonicity of $\varphi_\lambda$
and values of $\frac{1}{n}\varphi_\lambda(t)$,
$t\in(0,1)$, are close to one for $\lambda\in\CD_n(\eps)\cap\CP(n)$ 
since $n=\sum_k kR_k$ $\mu^{(n)}$-almost sure: if $t\in(0,1)$ and 
$\lambda\in\CD_n(\eps)\cap \CP(n)$ then
\[
\varphi_\lambda(t)=n-\sum_{k\ge2}kR_k(\lambda)\ge n-2\sum_{k\ge 2}R_k\ge n(1-2\eps).
\]

Let now $x_n$ be defined by~\eqref{eq:defxn}; taking $\alpha_x=1$
allows us to write 
$\CD_n(\eps)=\{\lambda\in\CP\colon\allowbreak \tphi_{x_n;\lambda}(t)<\eps\text{ 
for }t\in(1,2)\}$.  Let us estimate 
$\mu_{x_n}$ of the complement of this set. 
Take $u\in(0,\log(1/\rad))$ and apply Markov's inequality:
\[
\mu_{x_n}\left(\CP\setminus\CD_n(\eps)\right)=\mu_{x_n}\left\{\lambda\colon 
\exp\left(u\sum\nolimits_{k\ge2}R_k\right)\ge e^{u\eps n}\right\}
\le e^{-u\eps n}\prod_{k\ge 2}\frac{f(x_n^{\,k}e^u)^{b_k}}{f(x_n{\,^k})^{b_k}}.
\]
The product converges by the choice of $u$ and is bounded as $n\to\infty$
(it is checked like it was done in the proof of Lemma~\ref{lem:ExN}).

Let $m$ be the order of pole of $f$ (and $F$) 
at~$\rad=\rad_1$. The Laurent series decomposition~\eqref{eq:laurentF}
and expression \eqref{eq:laurentexN} for the mean of $N$ yields the
asymptotic relation $\rad-x_n\sim m\rad/n$. Hence $x_n^{\,n}\sim \rad^n e^{-m}$ 
and
\[
\mu_{x_n}\CP(n)
=\frac{a_n x_n^{\,n}}{F(x_n)}
\sim \frac{a_n \rad^n e^{-m}n^m}{|c_{-m}|(m\rad)^m},\qquad\qquad n\to\infty.
\]
Our next goal is to find a lower bound for $a_n$. It follows 
from \eqref{eq:defmun} that $a_n$ equals the sum of products in the right-hand
part of~\eqref{eq:defmun} over all partitions $\lambda\in\CP(n)$. Thus the sum
of the same product only over ``hook'' partitions $(n-j,1,1,\dots,1)$, 
$j=0,\dots,n-1$, gives a lower bound for $a_n$. 
The hypothesis $b_k\ge b>0$ allows us to find a bound
for this sum. Indeed,  the first Taylor coefficient $g_{k,1}$ of $f(z)^{b_k}$ 
is positive and moreover $g_{k,1}=g_{1,1}b_k\ge g_{1,1}b$ hence
\begin{equation}\label{eq:psce1}
a_n\ge\sum_{j=0}^{n-1} g_{1,j}g_{n-j,1}\ge g_{1,1}b\sum_{j=0}^{n-1}g_{1,j}\,.
\end{equation}
In order to find a lower bound of the partial sum of Taylor coefficients of
$f(x)$ we use the Hardy--Littlewood--Karamata theorem, see, e.~g.,
\cite[Thm.~XIII.5.5]{Feller2}, which states that
\[
g_{1,0}+\rad g_{1,1}+\dots+\rad^{n-1}g_{1,n-1}\sim \frac{c'_{-m}n^m}{\rad^m m!},
\qquad n\to\infty,
\]
where $c'_{-m}$ is the leading coefficient in the Laurent series for $f$ 
at $\rad$. Applying the same result with $n$ replaced by 
$\intfloor{n(1-\delta)}$, $\delta>0$, we obtain
\[
\rad^{\intfloor{n(1-\delta)}}g_{1,\intfloor{n(1-\delta)}}+\dots
+\rad^{n-1}g_{1,n-1}\sim \frac{c'_{-m}n^m(1-(1-\delta)^m)}{\rad^m m!}\,.
\]
Thus for large~$n$
\[
\frac{a_n}{g_{1,1}b}\ge \sum_{j=0}^{n-1}g_{1,j}
\ge \sum_{j=\intfloor{n(1-\delta)}}^{n-1}g_{1,j}
\ge \sum_{j=\intfloor{n(1-\delta)}}^{n-1}g_{1,j}\rad^{j-\intfloor{n(1-\delta)}}
\ge c_1 n^m\rad^{-n(1-\delta)}
\] 
for some $c_1>0$.  Combining the above estimates we see that
\[
\mu^{(n)}\left(\CP(n)\setminus\CD_n(\eps)\right)
=\frac{\mu_{x_n}\left(\CP(n)\setminus\CD_n(\eps)\right)}{\mu_{x_n}\CP(n)}
\le \frac{\mu_{x_n}\left(\CP\setminus\CD_n(\eps)\right)}{\mu_{x_n}\CP(n)}
\le c_2n^{-2m}e^{(\delta|\log\rad|-\eps u)n}
\]
for some $c_2>0$. Taking $\delta$ and $u$ such that the exponent is negative
shows that Conjecture~\ref{conj:sceerg} holds.
\end{proof}

\section{Examples}\label{sec:examples}
In this section we introduce three examples of families of multiplicative 
measures.  They are obtained from the well-known measures by distinct
deformations.  These deformations can be combined to produce another
examples, and also a different
measure can be taken as a starting point.

\subsection*{Weighted partitions}
Let us consider the measures $\mu^{(n)}$ which are proportional to some
constant $y>0$ to the power of the number of summands in partition. 
It corresponds to the following decomposition~\eqref{eq:Fasprod} of $F$:
\[
F(x)=\prod_{k=1}^\infty \frac{1}{1-yx^k}\,.
\]
 Similar measures were considered in~\cite{VY-MMJ}. 
If $y\le 1$ the convergence radius $\rad_1\ge1$ and the limit shape
is defined by 
\[
\varphi(t)=\frac{-\log(1-ye^{-t})}{\operatorname{Li}_2 y}
\]
with scaling $\alpha^{(n)}=\sqrt{n/\operatorname{Li}_2 y}$
where the dilogarithm $\operatorname{Li}_2 y$ is the normalizing factor.
Taking $y=1$  makes all weights equal and leads to the uniform measures on
partitions.  In this case it is more natural to take symmetric scaling
$\alpha^{(n)}=\sqrt{n}$ which leads to the celebrated limit shape
for the uniform measure on partitions defined by
\[
e^{-c\varphi(t)}+e^{-ct}=1,\qquad\qquad c=\frac{\pi}{\sqrt{6}}\,,
\]
found in \cite{T,ST,V-FAA} as mentioned in the Introduction.

If $y>1$ there is no limit shape in the grand canonical ensemble of partitions:
the distribution of $N$ is asymptotically equivalent to that of $R_1$, so
taking scaling $\alpha_x=1$ leads to the scaled Young diagram close to
the rectangle of unit width and random (asymptotically exponentially 
distributed) height.  In the small canonical ensemble, however, there
is a degenerate ergodicity, as follows from Proposition~\ref{prop:sceerg}
and can be also easily shown combinatorially.  
With the same scaling $\alpha^{(n)}=1$ 
the scaled Young diagram looks like the unit square, i.e.\ ``almost all''
parts in ``almost all'' partitions are ones, and larger parts do not
comprise a notable ratio to the weight, in the asymptotic sense.

\subsection*{Partitions with restricted part sizes}
Another possibility is to take $b_k=\mathbf{1}(k\in\mathcal{S})$ for a certain 
set $\mathcal{S}$ of positive integers.  This choice of $b_k$ 
makes $\mu^{(n)}$ the uniform measure on partitions of~$n$ with
all parts from $\mathcal{S}$. 

The distribution of the number of parts in such partitions has been
studied recently in~\cite{GH} under some assumptions on growth of
$B_k$.   Namely its is shown that if $B_k-ck^\beta$, 
$\beta\in(0,1)$, satisfies some additional condition then the number of parts
in a random partition of $n$ behaves like a nondegenerate random variable
(explicitly specified in~\cite{GH}) multiplied by $n^{1/(1+\beta)}$.
Theorem~\ref{thm:lceerg} shows that if just summands greater than
$tn^{1/(\beta+1)}$, $t>0$ are counted then their number is much less: it is 
proportional to $n^{\beta/(1+\beta)}$ and the coefficient converges in
probability to a constant (depending on~$t$).  It means that in this case 
a generic partition has plenty of small summands which do not contribute
a notable part to the whole sum.  This is related to a physical effect
known as Bose--Einstein condensation, see~\cite{V-UMN}.

\subsection*{Permutations with marked cycles}
As it was mentioned in the Introduction, the uniform measure on permutations 
induces a multiplicative measure on partition by considering partition on
cycle lengths. It is defined by decomposition
\[
F(x)=\frac{1}{1-x}=\prod_{k=1}^\infty e^{x^k/k}
=\prod_{k=1}^\infty \left(e^{x^k}\right)^{1/k}
\]
and hence satisfies~\eqref{eq:ourF} with $b_k=1/k$ 
but not \eqref{eq:Bk} since $\beta=0$.
Taking different $b_k$ in a form $b_k=c_k/k$ with integer $c_k$ 
corresponds to marking cycles of length $k$ in one of $c_k$ ways. In
particular, taking $c_k=k$  can be interpreted as choosing the first
element in each cycle, or, in the other words, making a set of ordered
lists from a permutation.  The numbers of such objects form
sequence A000262 in~\cite{OEIS}.  Other examples of combinatorial 
If one does not insist on a combinatorial
interpretation, it is possible to take real $c_k$, 
say, $c_k=\theta k^\beta$ for
$\beta,\theta>0$. It leads to the fulfillment of the condition~\eqref{eq:Bk} (and 
even~\eqref{eq:Bkdetailed}) with $\ell(k)=\theta/\beta+O(k^{-1})$.  Under this 
assumptions taking the scaling function 
$\alpha^{(n)}=(\theta \Gamma(\beta+1))^{-1/(\beta+1)}n^{1/(\beta+1)}$
leads to the limit shape
\[
\varphi(t)=\frac{\Gamma(\beta+1,t)-te^{-t}}{\beta \Gamma(\beta+1)}
\]
where $\Gamma(\beta,t)=\int_t^\infty u^{\beta-1}e^{-u}du$ is the incomplete
Gamma function.  This limit shape and fluctuations around it were found under
different (but overlapping with ours) assumptions in~\cite{EG08}.

If $\beta=\theta=1$ (i.e.\ the measure is induced by the uniform measure on
partitions of the set $\{1,\dots,n\}$ into ordered lists) the limit shape is the
exponent function ($\varphi(t)=e^{-t}$) in the scaling $\alpha^{(n)}=\sqrt{n}$.
Farther, formally letting $\beta\searrow 0$ 
and keeping $\theta$ fixed we approach the Poisson--Dirichlet distribution
$\mathcal{PD}(\theta)$.  However the limit shape becomes degenerate
(infinity at 0 and zero at $t>0$). It reflects the nonergodicity of the
limiting distribution.

\section*{Acknowledgment}
The author thanks A.~Vershik for helpful discussion of
this and related subjects.

\end{document}